\documentclass[preprint,12pt]{elsarticle}
\usepackage[utf8]{inputenc}
\usepackage[T2A]{fontenc}
\usepackage[english]{babel}

\usepackage[centertags]{amsmath}
\usepackage{amsfonts}
\usepackage{amssymb}
\usepackage{amsthm}
\newtheorem{Theorem}{Theorem}
\newtheorem{Aufgabe}{Problem}
\newdefinition{Remark}{Remark}
\newdefinition{acknowledgments}{Acknowledgments}
\newcommand{\alert}[1]{#1}

\begin{document}

\begin{frontmatter}

\title{On conservative difference schemes for the many-body problem}
\author[jinr,rudn]{V. P. Gerdt}
\author[rudn]{M. D. Malykh}
\author[jinr,rudn]{L. A. Sevastianov}
\author[KaiLi]{Yu Ying}
\address[jinr]{Joint Institute for Nuclear Research,
Dubna, Moscow Region, Russia}
\address[rudn]{Department of Applied Probability and Informatics,
RUDN University, Moscow, Russia.}
\address[KaiLi]{
Department of Algebra and Geometry,
KaiLi University, China
}

\begin{abstract}
A new approach  to the construction of difference schemes of any order for the many-body problem that preserves all its algebraic integrals is proposed. We introduced  additional variables, namely, distances and reciprocal distances between bodies, and wrote down a system of differential equations with respect to coordinates, velocities, and the additional variables. In this case, the system lost its Hamiltonian form, but all the classical integrals of  motion of the many-body problem under consideration, as well as new integrals describing the relationship between the coordinates of the bodies and the additional variables are described by linear or quadratic polynomials in these new variables. Therefore, any symplectic Runge-Kutta scheme preserves these integrals exactly. The evidence for the proposed approach is given. To illustrate the theory, the results of numerical experiments for the  three-body problem on a plane are presented with the choice of initial data corresponding to the motion of the bodies along a  figure of eight  (choreographic test).
\end{abstract}




\end{frontmatter}



\section*{Introduction}

One of the main continuous models is a dynamic system described by an autonomous system of ordinary differential equations, that is, a system of equations of the form
\begin{equation}\label{introd:ode}
\frac{dx_i}{dt}=f_i(x_1, \dots, x_n), \quad i=1,2,\dots n,
\end{equation}
where $t$ is an independent variable commonly interpreted as time, $x_1, \dots, x_n$ are the coordinates of a moving point or several points. In practice, the right-hand sides $f_i$ are often rational or algebraic functions of coordinates $x_1, \dots, x_n$ or can be reduced to such form by a certain change of variables.

In the framework of the classical approach, the solutions of Hamiltonian dynamical systems are found in two stages: first, algebraic integrals of motion \cite [\S~164]{Whittaker} are determined and then, based on Liouville's theorem \cite[\S~148]{Whittaker}, the problem is reduced to the calculation of integrals of algebraic functions referred to as quadratures.
However, construction of these integrals is usually not enough to reduce differential equations to quadratures.

For example, the classical problem of $n$ bodies \cite{Marshal} consists in finding solutions to an autonomous system of ordinary differential equations
\begin{equation}
\label{eq:nbp}
m_i \ddot{\vec{r}}_i = \sum \limits_{j=1}^n \gamma \frac{m_im_j}{r_{ij}^3} \left(\vec r_j -\vec r_i \right), \quad i=1, \dots, n
\end{equation}
Here $\vec r_i$ is the radius vector of the  $i$-th body, $m_i$ is its masses, $r_{ij}$ is the distance between the $i$-th and $j$-th bodies, and $\gamma$ is the gravitational constant.

Let us for brevity denote the components of the velocity of the  $i$-the body as $\dot x_i=u_i$, $\dot y_i=v_i$, and $\dot z_i=w$ and combine them in  vector  $\vec v_i$. This problem has 10 classical integrals of motion:
\begin{enumerate}
\item the momentum conservation  (3 scalar integrals)
\[
\sum \limits_{i=1}^n m_i \vec v_i = \mbox{const}
\]
\item the angular momentum conservation (3 scalar integrals)
\[
\sum \limits_{i=1}^n m_i \vec v_i \times \vec r_i = \mbox{const}
\]
\item the center-of-mass inertial motion (3 scalar integrals)
\[
\sum \limits_{i=1}^n m_i \vec r_i  - t \sum \limits_{i=1}^n m_i \vec v_i = \mbox{const}
\]
\item the energy conservation (1 scalar integral)
\[
\sum \limits_{i=1}^n \frac{m_i }{2} |\vec v_i|^2 - \gamma \sum \limits_{i,j} \frac{m_i m_j}{r_{ij}}= \mbox{const}
\]
\end{enumerate}
Bruns \cite[\S~164]{Whittaker} proved that every algebraic integral of  motion in this system is expressed algebraically in terms of the above 10 classical integrals. Note, that these algebraic integrals do not provide reduction of the  problem to quadratures.

When it is not possible to integrate a dynamic system analytically, it is done numerically. Among the methods of numerical analysis, the finite difference method is the most universal.
The numerical solution of the systems of autonomous equations \eqref{introd:ode} is performed via a difference scheme. Every difference scheme describes a transition from the value of $x$ taken at some  time $t$ to the value of $x$ taken at the next moment in time $t + \Delta t $. Hereafter we denote these new values as $ \hat x $. Usually this relationship is described by an algebraic equation similar, in its form, to a differential equation. For example, an explicit Euler scheme is described by a difference scheme
\[
\hat x-x= f(x)\Delta t.
\]
However, much more sophisticated schemes can also be  used, including implicit ones, which have the form of algebraic equations that not solved with respect to $ \hat x $, and multi-stage schemes {\em which use additional variables}. \alert{By this reason, we do not restrict} ourselves to any essential conditions in regarding the form of difference scheme.

\section{Finite-difference schemes preserving all the algebraic integrals of motion }
\label{n:diff}

Let us start from the basic definition. A difference scheme for autonomous system \eqref{introd:ode} is said to preserve the integral
\[
g(x_1, \dots, x_n) =C
\]
of this system if the equality
\[
g(\hat x) = g(x)
\]
is valid for any non-special choice of  $x$. We call a scheme conservative if it preserves all the algebraic integrals of motion of the considered autonomous system.

Standard finite-difference schemes, and first of all explicit Euler and Runge-Kutta schemes, do not preserve the algebraic integrals of motion, but \alert{yield} an approximate solution that is close to the exact solution at sufficiently small values of the time step in the considered time interval
\cite{Wanner}.
However, after its discretization, the model can be changed qualitatively, e.g., \alert{because of  the violation of energy conservation or geometric relationships and sometimes the appearance of dissipation} (cf.\cite{Lubich,Casc-2019}).

The idea to construct difference schemes that precisely preserve the integrals of motion in dynamical systems arose in the late 1980s. In the works of Cooper
\cite{Cooper}
and Yu.B. Suris
\cite{Suris-1988,Suris-1990}
a large family of Runge-Kutta schemes was discovered that preserve all the polynomial second-order integrals of motion that do not depend on $ t $ explicitly in any dynamical system and  also on the  symplectic structure in Hamiltonian systems
\cite{Lubich,Sanz-Serna}.

\begin{Remark}
Strictly speaking, these studies are related  only to integrals that are not explicitly dependent on time, and when they are presented in the literature 
\cite{Wanner},
these results are often limited to the case of quadratic forms. In reality, these reservations are not needed. In fact, one can always introduce a new independent variable $ \tau $ and rewrite the system \eqref{introd:ode} in the form
\begin{equation}\label{eq:tau}
\left\{
\begin{aligned}
&\frac{dx_i}{d\tau}=f_i(x_1, \dots, x_n), \quad i=1,2,\dots n,\\
& \frac{dt}{d\tau}=1.
\end{aligned}
\right.
\end{equation}
Any integral of  motion of the original system, quadratic in $x_1,\dots,x_n,t$,  is also a quadratic integral of the system \eqref{eq:tau} that does not contain $ \tau $. If necessary, a linear transformation can be converted to a quadratic form. By Cooper's theorem, this integral is preserved by any symplectic Runge-Kutta scheme applied to \eqref{eq:tau}. Since the right-hand side of the last equation of the system \eqref{eq:tau} is identically equal to unity, in any Runge-Kutta scheme, the projections of the slopes $ \vec k_i $ on the $ t $- axis will be equal to $ 1 $, and, hence, a step  made with respect to $ t $ is equal to the step along $ \tau $. As a result, the transition from layer to layer will be carried out in accordance to the same formulas as govern the transition done according to the Runge-Kutta scheme with the same Butcher matrix, but already written for the original system \eqref{introd:ode}. Therefore, the quadratic integral is preserved in the calculations according to the symplectic Runge-Kutta scheme applies the original system \eqref{introd:ode}.
\end{Remark}

\begin{Remark}
If a dynamical system admits only transcendental integrals of motion, one can try to construct difference schemes preserving these integrals. For example, recently Wan et al.
\cite[\S~5.2.1]{Wan-2017}
proposed a finite difference scheme for the Lotka-Volterra system, preserving its transcendental integral. The schemes obtained in this way cannot be expressed in terms of arithmetic operations only; in particular, the scheme proposed by Wan et al. involves logarithms. Since we intend to apply computer algebra methods to analyze of difference schemes in future research, it is of fundamental importance for us to deal with  algebraic equations only. For this reason, we avoid discussion of the use of transcendental integrals.
\end{Remark}

In the many-body problem~\eqref{eq:nbp}, 9 out of 10 classical integrals are quadratic in the coordinates of the bodies, their velocities and time $t$. Thereby, symplectic schemes exactly preserve 9 from 10 integrals, but they do not preserve the energy integral.Thus, the discretization introduces dissipation in the system.

The three-body problem describes the motion of a planet and its satellite around the Sun, and in this situation two characteristic parameters appear that have the dimension of time: the (approximate) period of revolution of the planet around the Sun and the period of revolution of the satellite around the planet. One can hope to obtain reasonable results by taking a time step substantially smaller than these two periods. At the same time, \alert{the positions of bodies after many periods can be of interest}. In this case, \alert{we need in solving the problem of many bodies in a large time scale. From general considerations, one could hope that symplectic integrators would allow greater accuracy in calculations for large times}.  The advantages and drawbacks of symplectic integrals in application to the many-body system long-term behavior have been discussed many times
\cite{Hernandez-2019}.

The first finite-difference scheme for the many-body problem, preserving all classical integrals of motion, was proposed in 1992 by Greenspan  \cite{Greenspan-1992,Greenspan-1995-1,Greenspan-1995-2,Greenspan-2004}
and independently in somewhat different form by  Simo and Gonz{\'a}lez
\cite{Simo-1993,Graham}.
The Greespan's scheme is a kind of combination of the midpoint method and discrete gradient method
\cite{McLachlan,Christiansen}
preserving a number of symmetries of the initial problem
\cite{Greenspan-1992}.

In the 2000s, several general approaches to construction of difference schemes were proposed that preserve all algebraic integrals of motion. To our opinion, these methods can be divided into two large groups: the methods based on modification of the Runge-Kutta schemes, and the methods used certain transformations of the original differential system.
The first group includes the parametric symplectic partitioned Runge-Kutta method
\cite{Brugnano-2012,Wang-2012}, in which the Runge-Kutta method is considered with the order of approximation chosen such that there is a free parameter in the procedure of determining the Butcher matrix coefficients. This parameter can be chosen to preserve a chosen quadratic integral.

Naturally, for practical implementation, just this step is the most difficult part of the procedure. If the scheme is chosen to be simplectic, then all quadratic integrals are preserved automatically. For example, for the many-body problem~\eqref{eq:nbp} one can construct an implicit finite difference scheme of any order. On the contrary, it is possible to take an explicit Runge-Kutta scheme with one or several parameters and define these parameters so that given integrals of motion become preserved. The time step itself is often used as a parameter. In this way a few explicit schemes have been constructed preserving both quadratic and non-quadratic integrals
\cite{Buono-2002,Zhang-2020-2}.
The Hamiltonian boundary value methods
\cite{Brugnano-2010,Brugnano-2012-2,Brugnano-2012-3},
can be considered as belonging to the same group. In these methods, the initial system is approximated according to the Galerkin procedure and then the scalar products are approximated by finite sums. As a result, the obtained method resembles the Runge-Kutta method, which preserves the integral of energy not exactly, but with any predetermined order of accuracy
\cite[Th. 3]{Brugnano-2012-2}.

The invariant energy quadratization method (IEQ, proposed
recently   Yang et al. \cite {Yang-2016, Yang-2017}) can be considered to be a member of the second group.
Based on this method, for a number of Hamiltonian systems including the two-body problem (Kepler problem), Hong Zhang et al.
\cite {Zhang-2020}
constructed finite-difference schemes conserving the energy. The IEQ method suggests finding such a change of variables, after which the energy becomes a quadratic form, so that the standard simplectic Runge-Kutta schemes, written for this transformed system, conserve its energy. The procedure admits increasing the number of unknowns and the loss of non-Hamiltonian form of equations.

In the past, every effort was made to perform replacements that preserve the symplectic structure of the many-body problem. In particular, when studying the simple collision of two bodies, the aim was to find not just a regularizing transformation, but a canonical transformation. Due to this,  the famous Weierstrass theorem on simple collisions was considered as an obstacle
\cite{Siegel}.
There is, however, a simple way to find a regularizing transformation, which was proposed by Burdet
\cite{Burdet}
and Heggie
\cite {Heggie}
and is thoroughly described in the book by Marchal
\cite[ch.~ 6] {Marshal}.
For this purpose, the simplectic structure of the problem should be abandoned and additional variables can be introduced. Therefore, the construction proposed by Hong Zhang et al.
\cite {Zhang-2020}
within the IEQ method for the two-body problem seems to be quite natural for the considered class of problems.

One of the delicate features of IEQ method is the preservation of constraints. The fact is that when additional variables are introduced, equations describing the relation between the coordinates and velocities of bodies, on the one hand, and the auxiliary variables, on the other hand. In the IEQ method these equations are not necessarily quadratic, so that these constraints are not preserved exactly. \alert{We argue that it is possible if one looks at such a change of variables that makes all integrals of motion and all constraint equations to be described by quadratic functions.} Therefore, in this paper we want to present the solution of the following problem.

\begin{Aufgabe}
\label{a:2}
For the $n$-body problem  \eqref{eq:nbp}
construct another system containing additional variables and having the following properties:
\begin{enumerate}
\item It possesses a sufficient number of algebraic integrals of motion in order to express additional variables in terms of the coordinates of the bodies,
\item With some choice of constant values in these integrals, its solutions coincide with those of the original system,
\item It has integrals of motion, which, if one takes into account the relationship between additional variables and the coordinates of the bodies, are transformed into 10 classical integrals of the many-body problem,
\item All integrals of motion of the system are quadratic in the coordinates and velocities of bodies, as well as in additional variables.
\end{enumerate}
\end{Aufgabe}

\section{Rationalization of the $n$-body problem}

First of all, let us eliminate irrationality by introducing new variables  $r_{ij}$, related to the coordinates by equation
\[
r_{ij}^2-(x_i-x_j)^2-(y_i-y_j)^2-(z_i-z_j)^2=0
\]
Differentiating this relation with respect to  $t$, we get
\[
r_{ij} \dot r_{ij} = (x_i-x_j)(u_i-u_j) +(y_i-y_j)(v_i-v_j)+(z_i-z_j)(w_i-w_j)
\]
or
\[
r_{ij} \dot r_{ij} = (\vec r_i - \vec r_j) \cdot (\vec v_i -\vec v_j)
\]
Let us rewrite the differential system~\eqref{eq:nbp} of $3n$ second-order equations into the  system of the first-order equations in the variables
\[
x_1, \dots, z_n, u_1, \dots, w_n, r_{12}, \dots, r_{n-1,n}
\]
This system consists of three coupled subsystems:
\begin{eqnarray}
&& \dot{\vec{r}}_i = \vec v_i,  \quad i=1, \dots, n \label{eq:nbp:1} \\
&& m_i \dot{\vec{v}}_i = \sum \limits_{j=1}^n \gamma \frac{m_im_j}{r_{ij}^3} \left(\vec r_j -\vec r_i \right), \quad i=1, \dots, n                                                   \label{eq:nbp:2} \\
&& \dot r_{ij} =\frac{1}{r_{ij}} (\vec r_i - \vec r_j) \cdot (\vec v_i -\vec v_j), \quad i,j=1, \dots, n; \,i \not =j. \label{eq:nbp:3}
\end{eqnarray}

Every solution to the many-body problem~\eqref{eq:nbp} satisfies this system, but, in general, the converse is not true. Not every solution to~\eqref{eq:nbp:1}-\eqref{eq:nbp:3} satisfies the relation
\[
r_{ij}^2=(x_i-x_j)^2+(y_i-y_j)^2+(z_i-z_j)^2
\]
and, thus, the systems~\eqref{eq:nbp} and~\eqref{eq:nbp:1}-\eqref{eq:nbp:3} have distinct solution sets. Therefore, generally speaking, the new system may have fewer integrals than the original. Howeverm, in the given case, we can prove that the new system inherits the set of classical integrals of the many-body problem~\eqref{eq:nbp}.

\begin{Theorem}
System \eqref{eq:nbp:1}-\eqref{eq:nbp:3} has 10 classical integrals of the problem~\eqref{eq:nbp} and additionally the integrals
\[
r_{ij}^2-(x_i-x_j)^2-(y_i-y_j)^2-(z_i-z_j)^2=\mbox{const}.
\]
\end{Theorem}

\begin{proof}
To check the conservation of expressions indicated in the theorem we will calculate the derivations of these expressions.

\begin{itemize}
\item The momentum conservation
\begin{equation}\label{momentum}
\frac{d}{dt} \sum \limits_{i=1}^n m_i \vec v_i =
\sum \limits_{i=1}^n m_i \dot{\vec{v}}_i = \sum \limits_{i=1}^n   \sum \limits_{j=1}^n \gamma \frac{m_im_j}{r_{ij}^3} \left(\vec r_j -\vec r_i \right) =0.
\end{equation}
\item The angular momentum conservation
\begin{equation}\label{AngularMomentum}
\frac{d}{dt} \sum \limits_{i=1}^n m_i \vec v_i \times \vec r_i  =
\sum \limits_{i=1}^n m_i \dot{\vec{v}}_i  \times \vec r_i = \sum \limits_{i=1}^n   \sum \limits_{j=1}^n \gamma \frac{m_im_j}{r_{ij}^3} \left(\vec r_j -\vec r_i \right) \times \vec r_i =0.
\end{equation}
\item The center-of-mass inertial motion
\begin{equation}\label{CenterMass}
\frac{d}{dt} \sum \limits_{i=1}^n m_i \vec r_i  - t \sum \limits_{i=1}^n m_i \vec v_i  = \sum \limits_{i=1}^n m_i \vec v_i  -  \sum \limits_{i=1}^n m_i \vec v_i  =0.
\end{equation}
\item The energy conservation
\[
\begin{aligned}
\frac{d}{dt} \sum \limits_{i=1}^n \frac{m_i }{2} (u_i^2+v_i^2+w_i^2)  & = \sum \limits_{i=1}^n m_i  \vec v_i \cdot \dot{\vec{v}}_i =   \sum \limits_{i=1}^n  \sum \limits_{j=1}^n \gamma \frac{m_im_j}{r_{ij}^3} \vec v_i \cdot \left(\vec r_j -\vec r_i \right)
\\ & =\gamma \sum \limits_{ij} \frac{m_im_j}{r_{ij}^3} \left(\vec v_i \cdot \left(\vec r_j -\vec r_i \right) + \vec v_j \cdot \left(\vec r_i -\vec r_j \right)\right)
\\ & =\gamma \sum \limits_{ij} \frac{m_im_j}{r_{ij}^3} \left(\vec v_i -\vec v_j \right) \cdot \left(\vec r_j -\vec r_i \right)
\end{aligned}
\]
and, due to \eqref{eq:nbp:3},
\[
\begin{aligned}
\frac{d}{dt} \sum \limits_{i,j} \frac{m_i m_j}{r_{ij}} = - \sum \limits_{i,j} \frac{m_i m_j}{r_{ij}^2} \dot r_{ij} = - \sum \limits_{i,j} \frac{m_i m_j}{r_{ij}^3} (\vec r_i - \vec r_j) \cdot (\vec v_i -\vec v_j)
\end{aligned}
\]
\item The additional conservation laws
\begin{equation}\label{additional}
r_{ij}^2-(x_i-x_j)^2-(y_i-y_j)^2-(z_i-z_j)^2=\mbox{const}
\end{equation}
follow from equations \eqref{eq:nbp:3}, since the derivative
\[
\frac{d}{dt} \left(r_{ij}^2-(x_i-x_j)^2-(y_i-y_j)^2-(z_i-z_j)^2 \right)
\]
is equal to
\[
2 r_{ij} \dot r_{ij} - 2 (\vec r_i -\vec r_j) (\vec v_i -\vec j) =0.
\]
\end{itemize}
\end{proof}


Now we have an autonomous system with rational right-hand side, all integrals of which are quadratic polynomials, except the following rational expression
\[
\sum \limits_{i=1}^n \frac{m_i }{2} (u_i^2+v_i^2+w_i^2) - \gamma \sum \limits_{i,j} \frac{m_i m_j}{r_{ij}}= \mbox{const}
\]
which corresponds to the energy conservation.

\section{System with quadratic polynomial integrals}

To get rid of the denominators in the energy conservation law, we introduce new additional variables $ \rho_ {ij} $ such that
\[
r_ {ij} \rho_ {ij} = 1.
\]
Note that this relation again is a quadratic polynomial what allow us to obtain a quadratic polynomial integral instead of the rational one.

Differentiating this relation with respect to $ t $, we obtain
\[
r_{ij} \dot \rho_{ij} + \rho_{ij} \dot r_{ij} =0
\]
or
\[
\dot \rho_{ij} = - \frac{\rho_{ij}}{r_{ij}^2} (\vec r_i - \vec r_j) \cdot (\vec v_i -\vec v_j).
\]
Now we rewrite the system \eqref {eq:nbp:1} - \eqref {eq:nbp:3} as the system of first-order equations with respect to the extended set of unknowns
\[
x_1, \dots, z_n, u_1, \dots, w_n, r_{12}, \dots, r_{n-1,n},  \rho_{12}, \dots, \rho_{n-1,n}
\]
Now we obtain the system of four coupled subsystems
\begin{eqnarray}
&& \dot{\vec{r}}_i = \vec v_i,  \quad i=1, \dots, n \label{eq:rho:1} \\
&& m_i \dot{\vec{v}}_i = \sum \limits_{j=1}^n \gamma \frac{m_im_j \rho_{ij}}{r_{ij}^2} \left(\vec r_j -\vec r_i \right), \quad i=1, \dots, n \label{eq:rho:2} \\
&& \dot r_{ij} =\frac{1}{r_{ij}} (\vec r_i - \vec r_j) \cdot (\vec v_i -\vec v_j), \quad i,j=1, \dots, n; \,i \not =j \label{eq:rho:3} \\
&&\dot \rho_{ij} = - \frac{\rho_{ij}}{r_{ij}^2} (\vec r_i - \vec r_j) \cdot (\vec v_i -\vec v_j), \quad i,j=1, \dots, n; \,i \not =j \label{eq:rho:4}
\end{eqnarray}
\alert{As well as the system~\eqref{eq:nbp:1}-\eqref{eq:nbp:3} the dynamical system system~\eqref{eq:rho:1}-\eqref{eq:rho:4} inherits the conservation laws of~\eqref{eq:nbp}.}

\begin{Theorem}\label{th:2}
The system \eqref{eq:rho:1}-\eqref{eq:rho:4} has 10 classical integrals of the many-body problem~\eqref{eq:nbp}, namely, Eqs.\,\eqref{momentum}-\eqref{CenterMass},
the energy conservation in the form
 \[
\sum \limits_{i=1}^n \frac{m_i }{2} (u_i^2+v_i^2+w_i^2) - \gamma \sum \limits_{i,j} m_i m_j \rho_{ij}= \mbox{const}\,,
\]
the additional integrals~\eqref{additional} and
\begin{equation}\label{eq:int:rho}
r_{ij}\rho_{i,j}=const,\quad i\neq j.
\end{equation}
\end{Theorem}

\begin{proof}
\alert{In perfect analogy to the proof of Theorem 1 we perform explicit differentiation and simplify the obtained expressions. For Eqs.\,\eqref{momentum}-\eqref{additional} it is done exactly as the proof of Theorem 1. Since now the energy conservation has a slightly different form, we present the relevant computation}
\[
\begin{aligned}
\frac{d}{dt} \sum \limits_{i=1}^n \frac{m_i }{2} (u_i^2+v_i^2+w_i^2)  & = \sum \limits_{i=1}^n m_i  \vec v_i \cdot \dot{\vec{v}}_i =   \sum \limits_{i=1}^n  \sum \limits_{j=1}^n \gamma \frac{m_im_j \rho_{ij}}{r_{ij}^2} \vec v_i \cdot \left(\vec r_j -\vec r_i \right)
\\ & =\gamma \sum \limits_{ij} \frac{m_im_j \rho_{ij}}{r_{ij}^2} \left(\vec v_i \cdot \left(\vec r_j -\vec r_i \right) + \vec v_j \cdot \left(\vec r_i -\vec r_j \right)\right)
\\ & =\gamma \sum \limits_{ij} \frac{m_im_j \rho_{ij}}{r_{ij}^2} \left(\vec v_i -\vec v_j \right) \cdot \left(\vec r_j -\vec r_i \right)
\end{aligned}
\]
and, in view of \eqref{eq:rho:4},
\[
\begin{aligned}
\frac{d}{dt} \sum \limits_{i,j} m_i m_j \rho_{ij} = - \sum \limits_{i,j} \frac{m_i m_j \rho_{ij}}{r_{ij}^2} (\vec r_i - \vec r_j) \cdot (\vec v_i -\vec v_j).
\end{aligned}
\]
The validity of the conservation law~\eqref{eq:int:rho} follows from equations \eqref{eq:rho:3}-\eqref{eq:rho:4}:
\[
\frac{d}{dt} r_{ij} \rho_{ij}= \dot r_{ij} \rho_{ij} +  r_{ij} \dot \rho_{ij} = \frac{\rho_{ij}}{r_{ij}}  (\vec r_i - \vec r_j) \cdot (\vec v_i -\vec v_j) -  r_{ij} \frac{\rho_{ij}}{r_{ij}^2}  (\vec r_i - \vec r_j) \cdot (\vec v_i -\vec v_j)  =0.
\]
Since the differential equations of the considered system were derived by differentiating relations \eqref{additional} and \eqref{eq:int:rho}, the appearance of the above additional integrals is obvious.
\end{proof}

In general, it is not possible to state that the system  \eqref{eq:rho:1}-\eqref{eq:rho:4} has no other algebraic integrals, functionally independent of those listed above. However, a solution to the many-body problem satisfies to the extended system \eqref{eq:rho:1}-\eqref{eq:rho:4}, so that any algebraic integral of motion of the extended system remains constant on solutions of the many-body problem.  By the Bruns theorem
\cite[\S~164]{Whittaker}
on a manifold
\[
r_ {ij} ^ 2- (x_i-x_j) ^ 2- (y_i-y_j) ^ 2- (z_i-z_j) ^ 2 = 0, \quad
r_ {ij} \rho_ {ij} = 1, \quad i \not= j,
\]
\alert{such} integral is expressed algebraically in terms of the classical integrals.

According to Theorem \ref{th:2}, the autonomous differential system \eqref{eq:rho:1} - \eqref{eq:rho:4} containing $ n (n-1) $ additional variables $ r_ {ij} $ and $ \rho_ {ij} $, has the following properties:
\begin{itemize}
\item It has the quadratic integrals of motion
\[
r_{ij}^2-(x_i-x_j)^2-(y_i-y_j)^2-(z_i-z_j)^2=\mbox{const}
\]
and
\[
r_{ij} \rho_{ij}=\mbox{const},
\]
which allow to express the additional variables $ r_ {ij} $ and $ \rho_ {ij} $ in terms of the coordinates of the bodies.
\item If the constants in these integrals are chosen in such a way that
\[
r_{ij}^2-(x_i-x_j)^2-(y_i-y_j)^2-(z_i-z_j)^2=0
\]
and
\[
r_{ij} \rho_{ij}=1,
\]
then solutions to system coincide with those to the initial one~\eqref{eq:nbp}.
\item The new system has quadratic integrals of motion, which, taking into account the relationship between additional variables and the coordinates of the bodies, turn into 10 classical integrals of the many-body problem.
\end{itemize}

\noindent
Thus, the constructed system posses all the properties listed in the Problem~\ref{a:2}.

\section{Conservative schemes for $N$ body problem}

Since all classical integrals of the many-body problem, as well as the additional integrals, are quadratic in their variables, any symplectic Runge-Kutta difference scheme, including the simplest midpoint one, that is 
\begin{equation}
\label{eq:mp}
\hat x - x =f \left(\frac{\hat x+x}{2}\right) dt,
\end{equation}
preserves all these integrals. In particular,  the midpoint scheme written for the system \eqref{eq:rho:1}-\eqref{eq:rho:4}, preserves all its algebraic integrals exactly and is invariant under permutations of bodies and time reversal. It is not difficult to create high-order schemes which preserve all integrals of motion in the many-body problem.

The autonomous system of differential equations \eqref{eq:rho:1} - \eqref{eq:rho:4} preserves the symmetry of the original problem with respect to permutations of bodies and time reversal.
As noted above, the  midpoint scheme is invariant under these symmetries.

At each step of the scheme, new values will be determined not only for the coordinates and velocities, but also for the auxiliary quantities $r_ {ij}$ and $\rho_ {ij}$. If at the initial moment only the coordinates and velocities of the bodies were specified and the auxiliary variables were defined by the equalities
\[
r_{ij}=\sqrt{(x_i-x_j)^2+(y_i-y_j)^2+(z_i-z_j)^2}, \quad
\rho_{ij}=\frac{1}{r_{ij}},
\]
then these equalities are preserved exactly (up to the signs of the radicals) due to the fact that the auxiliary integrals \eqref{additional} and \eqref{eq:int:rho} are quadratic and exactly preserved under  the action of midpoint scheme. Therefore, the quantities $ r_ {ij} $ and $ \rho_ {ij} $ remain the distances between bodies and the inverse distances between bodies.

\section{Numerical experiments with planar three-body problem}

We developed a numerical algorithm based on the midpoint scheme  \cite {YY-2019} and realized it in SageMath~\cite{sage,SageMath}.  Discussion of this algorithm and its program implementation is beyond the scope of this paper and  will be presented in other paper. Now we illustrate the above considerations by a numerical example. For simplicity, we consider the planar dimensionless  problem of the motion of three bodies of equal mass with $ m_i = \gamma = 1 $. 

Before starting the numerical experiments, it is necessary to discuss the issue of organizing the calculations by means of the midpoint scheme. The problem is that this scheme is implicit, so that at each step it is necessary to solve the system of algebraic equations  \eqref{eq:mp}. 

We are going to use the method of simple iterations to solve the system of nonlinear algebraic equations \eqref{eq:mp} with respect to $\hat x$: starting from  $\hat x^{(0)}=x $ the sequence 
\begin{equation}\label{eq:sec}
\hat x^{(n+1)} = x +  f\left(\frac{\hat x^{(n)} +x}{2}\right) dt, \quad n=0,1,2, \dots
\end{equation}
is constructed. If this sequence has a limit  $\hat x$, then it is a root of the system of equations \eqref{eq:mp}. In practice it is  inconvenient to use the known estimates for  $\|\hat x^{(n)}-\hat x\|$, since they are rather cumbersome, particularly, for systems of equations with a large number of unknowns. However, we use the midpoint scheme in order to preserve the integrals of motion, 
therefore, we will  monitor that the increment of the integrals of motion does not go beyond the given boundaries ($10^{-8}$ in our examples) rather than that $\|\hat x^{(n)}-\hat x\|$ is small. When conducting numerical experiments, we often encountered divergence of the method. In this case we reduce the step $dt$. 

\subsection{Test on Lagrange solutions}

\begin{figure}
\centering
\includegraphics[width=0.5\textwidth]{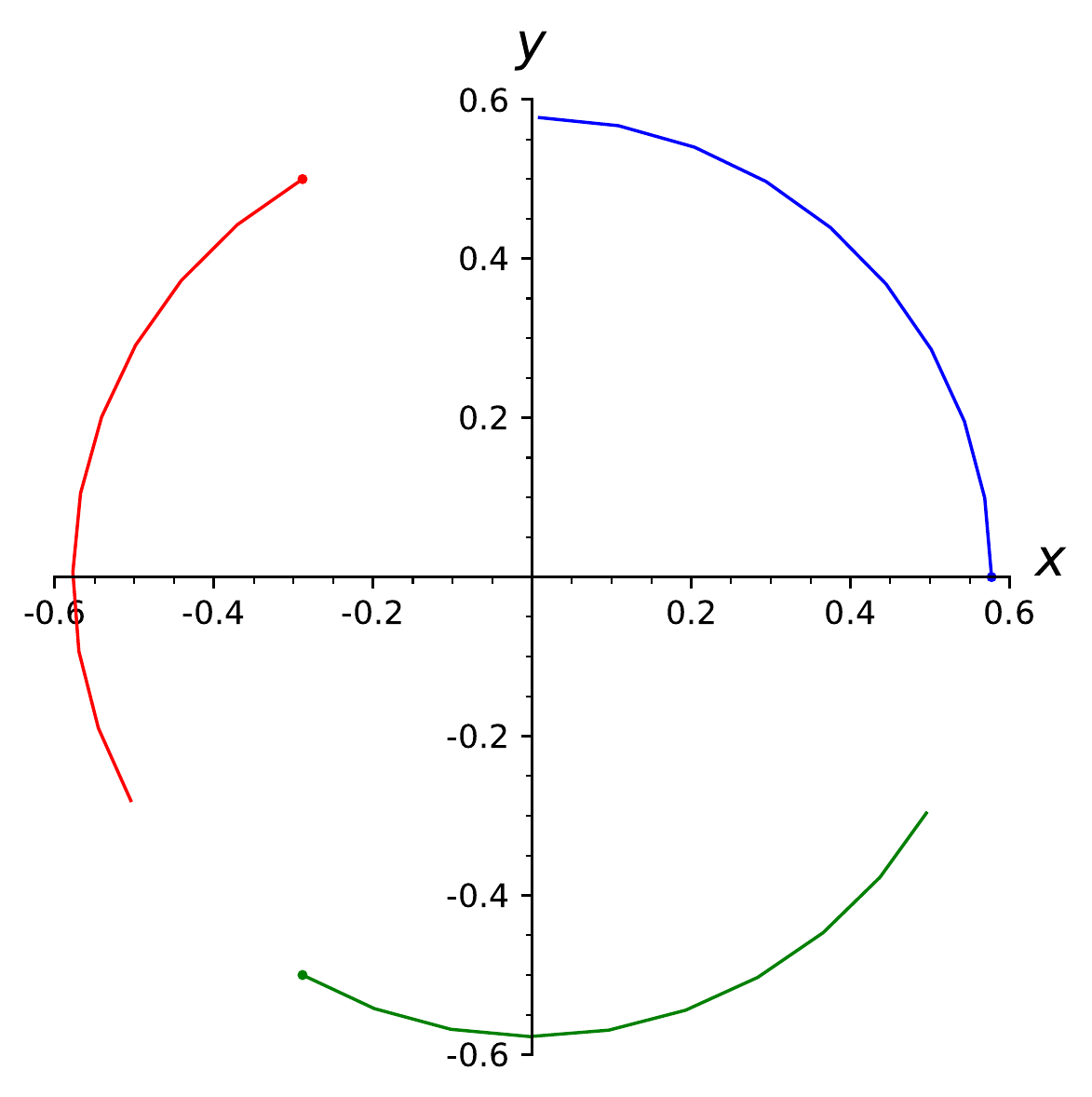}
\caption{Lagrange test, step $dt=1/10$. Trajectories of three bodies at  $0<t<1$, the initial positions are marked by dots.}
\label{fig:lagrange:0}
\end{figure}

\begin{figure}
\centering
\includegraphics[width=0.7\textwidth]{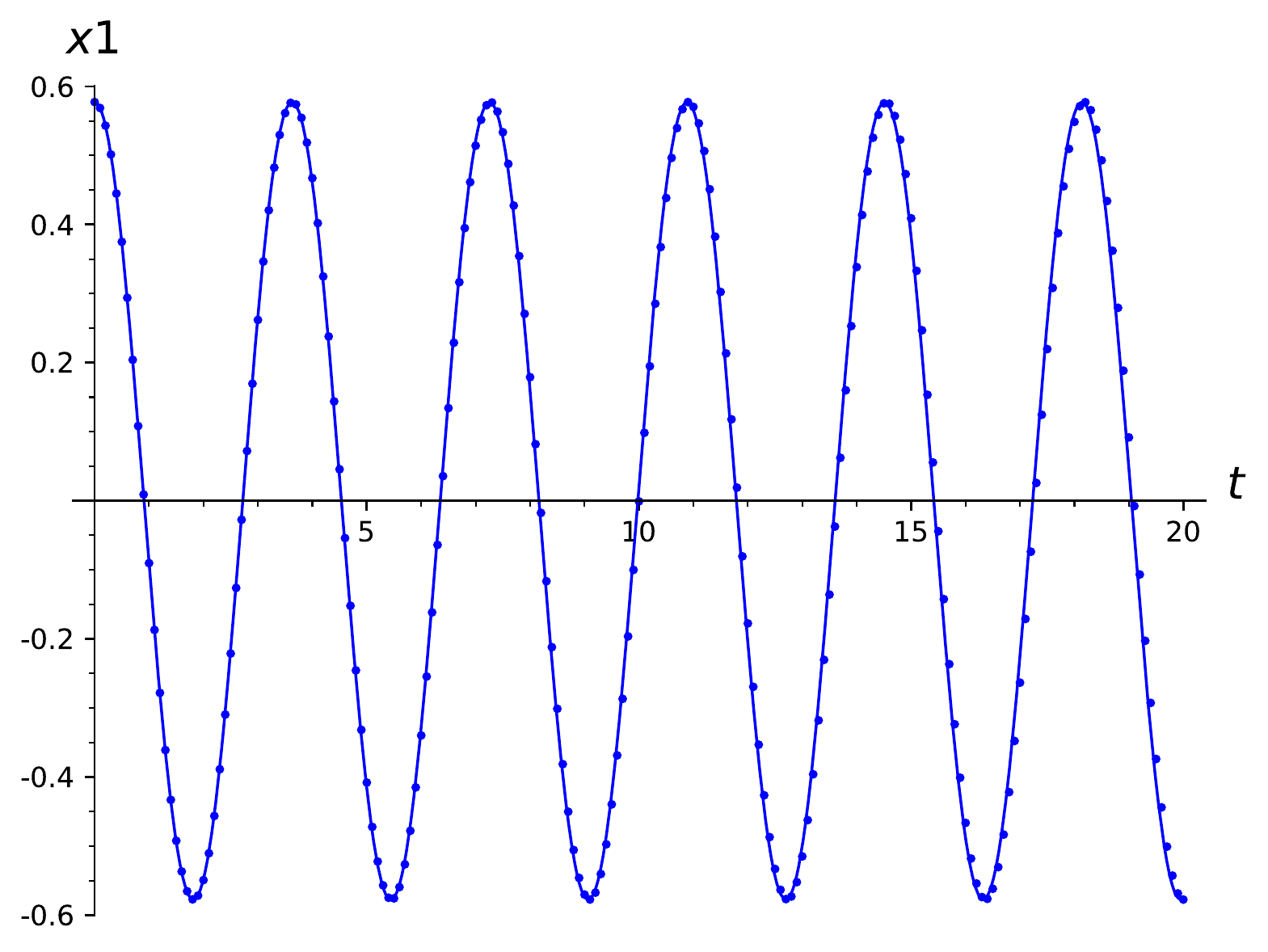}
\caption{Lagrange test, step  $dt=1/10$. Solid line shows the exact solution, dots represent the approximate one.}
\label{fig:lagrange:1}
\end{figure}

As is well known since Euler
\cite{Marshal,Montgomery-2001},
the planar three-body problem has two families of particular solutions that can be described in elementary functions. The first of them, traditionally associated with the name of Lagrange, can be described as follows: the bodies form a regular triangle with sides of a fixed length $a$, which rotates around the center of mass with constant speed. We took $ a = 1 $ and a quite noticeable step $ dt = 1/10 $, nevertheless, the nature of the movement remained the same, which is clearly visible along the trajectories (Fig. \ref{fig:lagrange:0}). The difference from the exact solution is barely noticeable in Fig. \ref{fig:lagrange:1}. 

At the same time, the distances between the points remain constant at the level of rounding error. The error in calculating the distance between the center of gravity and the bodies, as well as in determining the energy and angular momentum, is much larger, at the level of the specified $ 10^{-8} $. Thus, throughout the entire considered time interval, the bodies fall at the vertices of a regular triangle, which rotates around the center of mass.

\subsection{Choreographic test}

\begin{figure}
\centering
\includegraphics[width=0.9\textwidth]{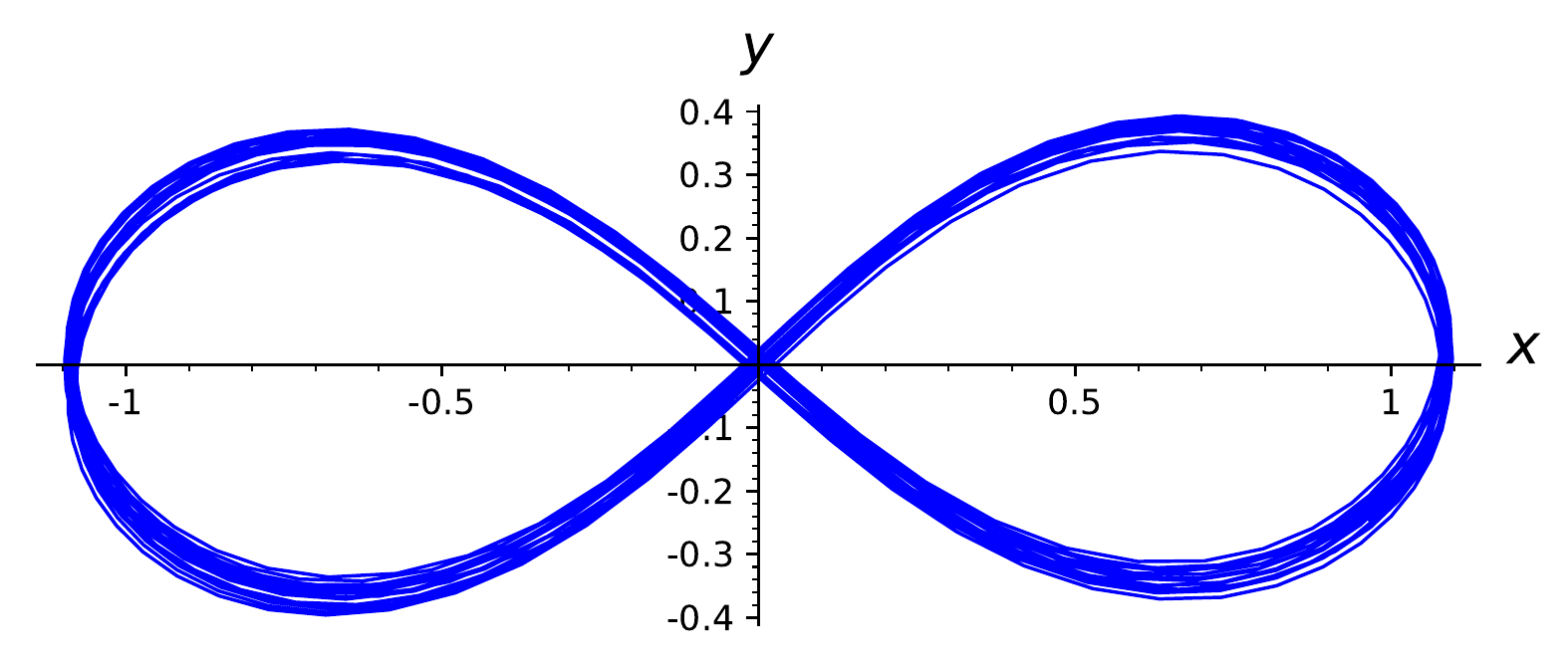}
\caption{Choreographic test, step $dt=1/10$. Trajectory of the first body for $0<t<200$, initial positions marked by dots.}
\label{fig:moore:1}
\end{figure}

\begin{figure}
\centering
\includegraphics[width=0.7\textwidth]{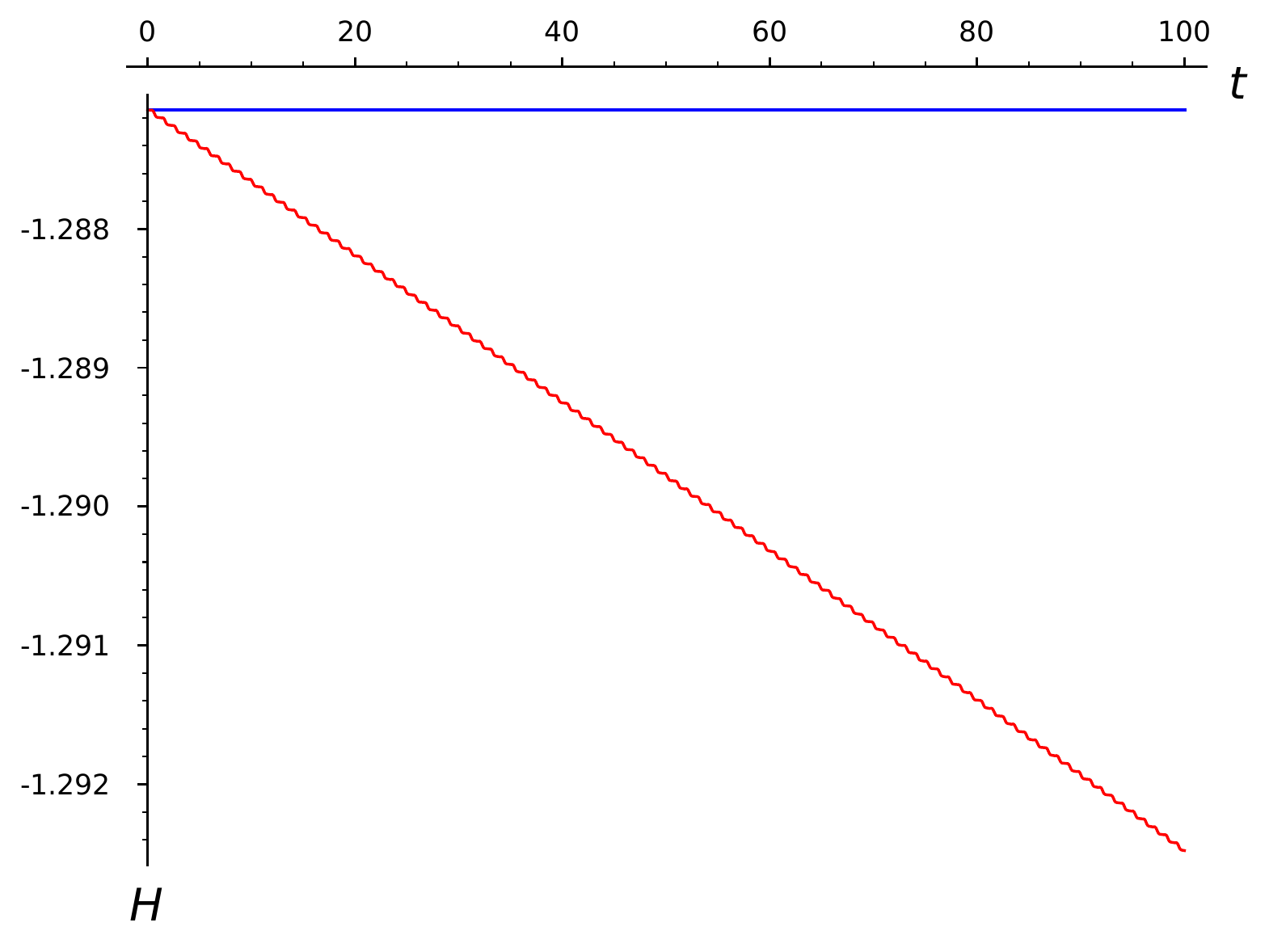}
\caption{Choreographic test, step $dt=1/10$. Dependence of energy $H$ on time for approximate solutions found using our algorithm (blue) and the rk4 scheme (red).}
\label{fig:simo:2}
\end{figure}

In 1993, Chris Moore discovered
\cite{Moor},
as part of a numerical experiment, a new periodic solution of the three-body problem, in which three bodies write out the eight; later Chenciner and Montgomery
\cite{Montgomery-2000,Montgomery-2001}
gave a justification to this fact. The initial conditions are known only approximately. We will use the values found numerically by Carles Simó and given in \cite{Montgomery-2000}.

First of all, we note that the figure eight really turns out (Fig.~\ref{fig:moore:1}). However, at a coarse step it is noticeable that this eight moves in time. It is not clear whether this movement can be stopped by the refinement of initial data.

To compare our algorithm and the standard algorithms based on explicit schemes, we took the fourth-order standard Runge-Kutta scheme (rk4) with the same coarse step $ \Delta t = 0.01$. Of course, the calculation according to this scheme is faster, and the approximation is better by two orders in magnitude. Nevertheless, the error in calculating the energy from rk4 grows almost linearly like $ 10^{- 4} t $ and becomes observable soon, and according to our scheme it does not leave the error corridor $\pm 10^{-8}$ (Fig.~\ref{fig:simo:2}).

\subsection{Test with small distances}
\label{n:col}

\begin{figure}
\centering
\includegraphics[width=0.5\textwidth]{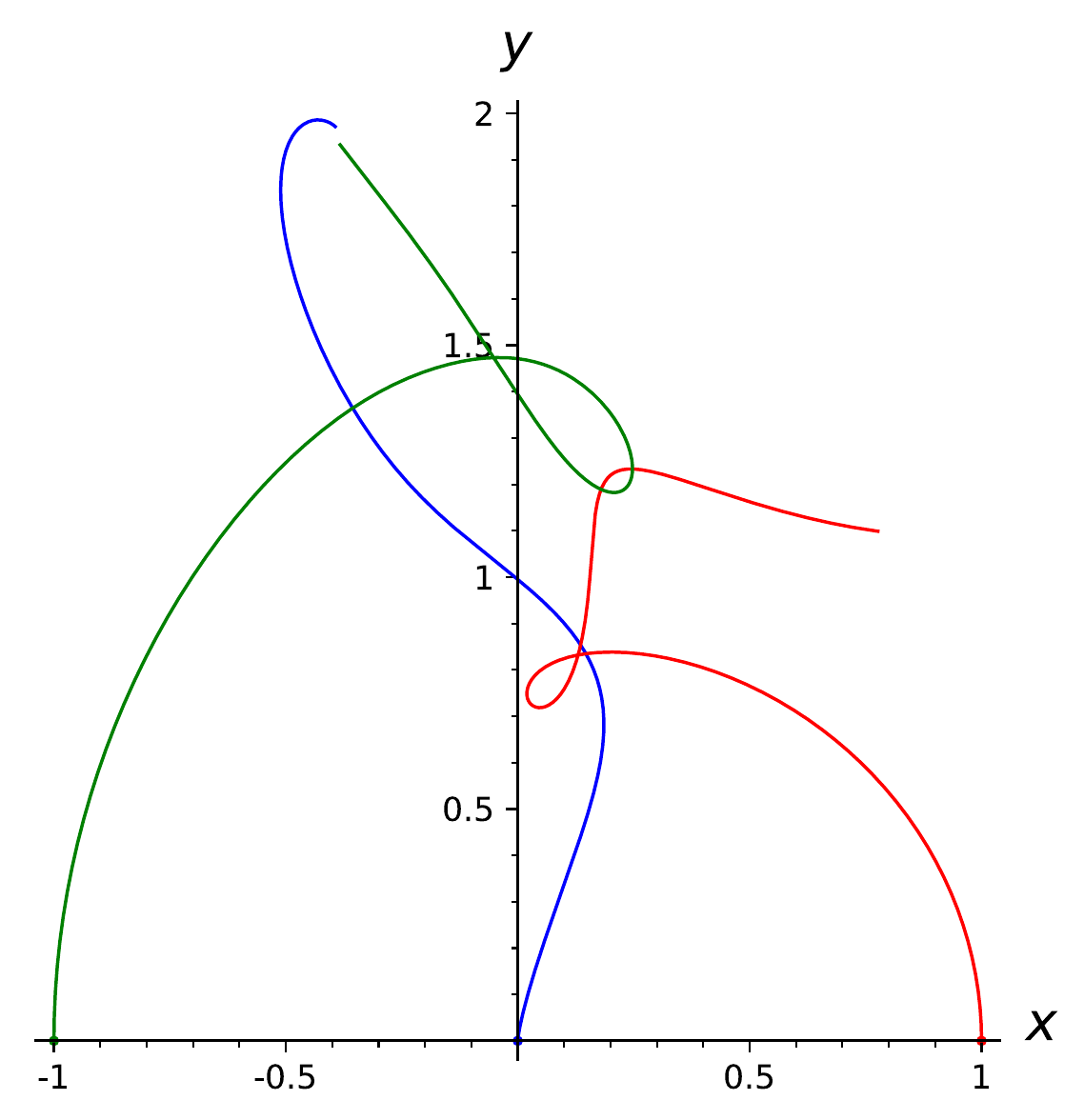}
\caption{Test with collision, step $dt=6\cdot 10^{-3}$. Trajectories of the three bodies for $0<t<2$, initial positions marked by dots.}
\label{fig:col:0}
\end{figure}

\begin{figure}
\centering
\includegraphics[width=0.5\textwidth]{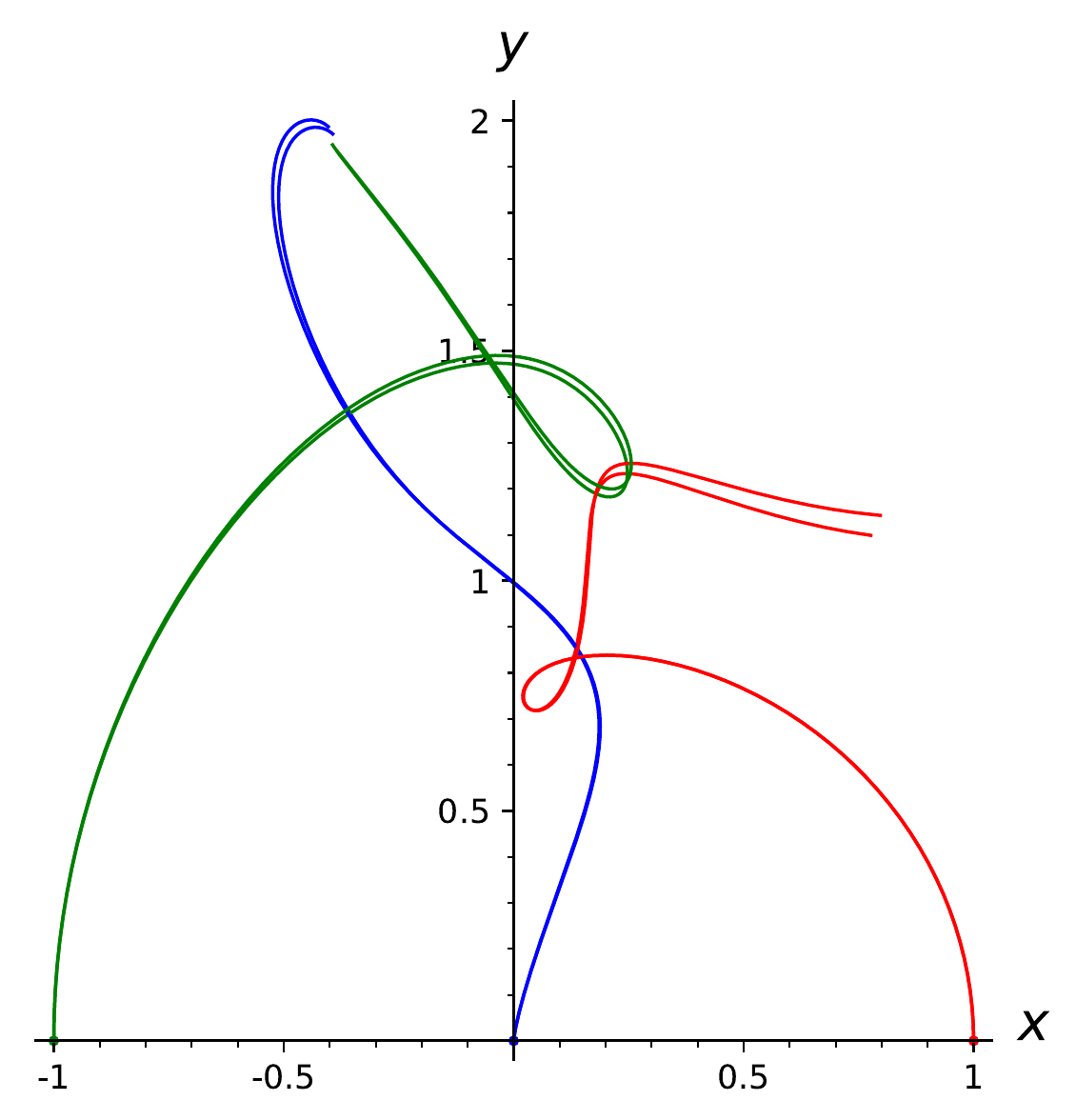}
\caption{Test with collision, step $dt=6\cdot 10^{-3}$. Trajectories of the three bodies for $0<t<2$, the initial positions are marked by dots.}
\label{fig:col:5}
\end{figure}

In all the tests considered above, the bodies remained at a considerable distance from each other, so there was no computational singularity rather typical for the many-body problem, i.e., the situation when  bodies pass  close to each other. 

At the initial moment of time, the following conditions are taken: the bodies lie on the line $ Ox $, the first body at the point $ x = 0 $, the second at  $ x = 1 $, the third at  $ x = - 1$. The first body is at rest, and the initial velocities of the other two bodies are directed along the $y$-axis and are equal to $ 1 $ and $ 1.5 $, respectively. Then  in the process of computation, it turns out that the bodies several times pass close to each other (Fig. \ref{fig:col:0}). Moreover, the applicability conditions of the iteration method make it necessary to reduce the step.  Of course, we cannot assume that at this moment a true collision occurs, since, according to the Weierstrass theorem, the collision in the many-body problem with arbitrary initial data is improbable \cite{Siegel}. 

The question of the stability of the many-body problem with respect to the initial data is very complex and we would not want to invade this area here. However, it seemed important to us to note that under small perturbations of the initial data our scheme yields trajectories close to the original ones. We checked this  in several numerical experiments, e.g., in the test with possible collision, we increased the velocity of the third body by $ 0.01 $ and compared the trajectories (Fig. \ref{fig:col:5}). It is clearly seen that the new trajectories are close to the old ones, and the point of possible collision turns out to be almost at the initial position.

\section*{Conclusion}

The major result of our work is Theorem \ref{th:2}, which yields a solution to Problem \ref{a:2} (Section \ref{n:diff}) and can be used to construct finite-difference schemes exactly preserving all its algebraic integrals and invariant under the permutations of bodies and time reversal. The midpoint scheme is an example of such scheme when it applied \eqref{eq:rho:1}-\eqref{eq:rho:4}. It should be  emphasized that in this way we can construct high-order schemes that preserve all integrals of motion in the many-body problem.

The key problem with the application of these schemes in practice, of course, is their implicitness, which can be overcome by both numerical and numerical-analytical methods. Obviously, errors in solving a nonlinear system of algebraic equations describing one step in a scheme can completely level out all the advantages of the proposed scheme~\cite{Zhang-2020-2}.
Therefore, in our opinion, this issue requires further study, including numerical experiments. The numerical example given above is intended only to illustrate theoretical results.

However we believe that the conservative difference schemes constructed above for the many-body problem will not only make it possible to carry out calculations for large times, but will also make it possible to qualitatively investigate the properties of solutions of the many-body problem by the finite difference method.

It should also be noted that when comparing the schemes, the focus of the researchers' attention was always on the quantitative proximity of the exact and approximate solutions, while the question of preserving the qualitative properties of the exact solution has remained insufficiently studied, although the  tempting possibility to ``determine the nature of the dynamic process using only rough calculations with a large grid pitch'' according to the symplectic scheme was noted in
\cite{Gevorkyan-2013}.
Meanwhile, for a linear oscillator, the simplest Runge-Kutta symplectic scheme, the midpoint scheme, allows not only  accurate preservation of the energy integral, but also  obtaining an exact periodic approximate solution and closed polygons as an analogue of closed phase trajectories of the continuous model
\cite{YY-2019}.

In other words, after discretization, a model is obtained that inherits almost all the algebraic and qualitative properties of a continuous model.
In the theory of partial differential equations, discretizations inheriting certain properties of differential equations and operations on the dependent variables included in them, are called {\it mimetic} (i.e., imitative) or {\it compatible}
\cite{Shashkov-1996,Arnold-2006,Castillo-2013,Vega-2014}.
Therefore, we find it appropriate to consider the construction of conservative difference schemes in the context of the more general and more important issue of constructing mimetic difference schemes for the many-body problem, of course, clarifying the concept itself.

Nevertheless, on the basis of the midpoint scheme and idea of introducing auxiliary variables for the many-body problem, a difference scheme is explicitly presented that preserves all the classical integrals of motion precisely and is invariant under permutations of bodies and time reversal. In further studies we plan to explore numerical and algebraic properties of this system. We believe that this scheme inherits at least a part of the algebraic and group-theoretical properties of exact solutions of the many-body problem.

\begin{acknowledgments}
This work is supported by the Russian Science Foundation (grant no.~20-11-20257).
\end{acknowledgments}

\section*{References}
\bibliographystyle{elsarticle-num}
\bibliography{malykh.bib}

\end{document}